\documentclass[12pt,reqno]{amsart}
\usepackage{amsmath}
\usepackage{amsaddr}   
\usepackage{relsize}
\usepackage{adjustbox}
\usepackage{graphicx}
\usepackage{amssymb} 
\usepackage[margin=2.5cm]{geometry}
\usepackage{float}

\newcommand{\Jimm}{\mathbf J}

\newcommand{\Q}{\mathbb Q}
\newcommand{\Real}{\mathbb R}
\newcommand{\Z}{\mathbb Z}

\newcommand{\C}{\mathbb C}

\newcommand{\stab}{\operatorname{Stab}}

\newcommand{\psl}{ \mathsf{PSL}_2 (\Z)    }

\newcommand{\pgl}{ \mathsf{PGL}_2 (\Z)   }

\newcommand{\pglc}{ \mathsf{PGL}_2 (\C)  } 
\newcommand{\pslr}{ \mathsf{PSL}_2 (\Real)   }

\newtheorem{theorem}{Theorem}[section]

\newtheorem{lem}[theorem]{Lemma}
\newtheorem{proposition}[theorem]{Proposition}
{ \theoremstyle{remark}\newtheorem*{remark}{Remark} }

\theoremstyle{remark}
\newtheorem{example}{{\rm\textbf{Example}}}

\title{Equivariant Modular Functions and Quantizations of Continued Fractions}

\author[Mustafa Topkara]{Mustafa~Topkara}
\address{Department of Mathematics, Mimar Sinan University, Istanbul, Türkiye}
\email{m.e.topkara@gmail.com}

\author[A. Muhammed Uludağ]{A.~Muhammed~Uludağ}
\address{Department of Mathematics, Galatasaray University, Istanbul, Türkiye}
\email{muludag@gsu.edu.tr}

\begin{document}

\maketitle

\section{Introduction}
\subsection{Equivariant functions}
Suppose that a group $G$ acts on the sets $X$ and $Y$ from the left.
We say that a function $\psi: X\to Y$ is equivariant with respect to these actions if
$$
\psi(gx)=g \psi(x) \quad (x\in X, \, g\in G).
$$
If $X$, $Y$ carry additional structures, $G<\mathrm{Aut}(X)$, 
and the $G$-action on $Y$ is defined via a homomorphism $\Psi:G \to  \mathrm{Aut}(Y)$, then the equivariance condition can be reformulated as
\begin{equation}\label{setting}
\psi(gx) = \Psi(g) \psi(x)  \quad (x\in X, \,  g\in G).
\end{equation}
We call the pair $(\Psi, \psi)$ an {\it equivariant pair}.

Observe that, by Condition~(\ref{setting}), if $x$ is fixed by $g$, then $\psi(x)$ is fixed by $\Psi(g)$.

\subsection{Morier-Genoud and Ovsienko quantization}
Let 
\begin{align*}
X:=&{\mathsf P^1}(\Z)=\Bigl\{[m:n] \, |\, m,n \in \Z, \quad  (m,n)\neq (0,0)\Bigr\}, \\
G:=&\mathsf{PSL}_2(\Z),\\
Y:=&{\mathsf P^1}(\Z[q])=\Bigl\{[A:B] \, |\, A, B  \in \Z[q],  \quad (A,B)\neq (0,0)\Bigr\},
\end{align*}
where $\Z[q]$ is the polynomial ring with integral coefficients and $\Z(q)$ is its quotient ring, the field of rational functions with integral (or equivalently $\Q$-) coefficients. Recall that
\begin{align*}
\psl:=\Bigl\{
&M:[m:n] \in \mathsf{P^1}(\Z) \mapsto [am+bn: cm+dn] \in {\mathsf P^1}(\Z)\, |\\
&\,a,b,c,d \in \Z, \quad  ad-bc=1\Bigr\},
\end{align*}
and set
\begin{align*}
\mathsf{PGL}_2(\Z(q)):=
\Bigl\{
&M:[m:n] \in \mathsf{P^1}(\Z[q]) \mapsto [Am+Bn: Cm+Dn]\in \mathsf{P^1}(\Z[q]) \, |\\ 
&\,A,B,C,D \in \Z[q], \quad AD-BC\neq 0\Bigr\}.
\end{align*}

It has been shown in  \cite{Valentin} (see also~\cite{genoud}) that a non-trivial equivariant pair $(\Psi, \psi)$ with 
\begin{align*}
\Psi:&\,\,\psl\to \mathsf{PGL}_2(\Z(q)),\\
\psi:&\,\,\mathsf{P^1}(\Z) \to \mathsf{P^1}(\Z[q]) 
\end{align*}
exists, which furthermore satisfies the extra `{quantization}' condition 
$$
\psi([m:1])=\psi(m)=\frac{1-q^m}{1-q} \quad (m=1,2,\dots).
$$
In particular, this requires $\psi(1)=1$.
The value $\psi([m:n])$ is called the \emph{quantization} of the rational $m/n$
and is denoted $\psi(x):=:[x]_q$. 
The representation $\Psi$ itself, which is faithful, is called the \emph{quantization} of  $\psl$. 

\subsection{Purpose of the paper}
We show that there exist exactly three equivariant pairs 
$(\Psi, \psi)$ with $\Psi:\psl\to\mathsf{PGL}_2(\C(q))$.
One of them is the pair $(\Psi, \psi)$  described above, with the image of $\Psi$ actually lying in 
 $\mathsf{PGL}_2(\Z[q])$.
In addition, there is a pair of conjugate equivariant pairs 
$(\Psi^\pm, \psi^\pm)$ with the image of $\Psi^\pm$ actually lying in 
 $\mathsf{PGL}_2(\Z[\omega][q])$, where $\omega=\exp(2\pi i/6)$.
Both representations $\Psi$  and $\Psi^\pm$ admit a natural and unique extensions to $\pgl$, and the maps
$\psi$  and $\psi^\pm$ are equivariant with respect to the $\psl$-action.

We also discuss some specializations of $q$. 
We show that, when $q=(-3\pm\sqrt5)/2$, 
the representation $\Psi$ is conjugate to Dyer's outer automorphism $\alpha$ of 
$\pgl$ and the quantization map $\psi$  is a translate of the involution 
$\Jimm$ discovered in \cite{conumerator} by a Möbius transformation. 
There is a similar result for the equivariant pairs $(\Psi^\pm, \psi^\pm)$.

\section{Quantization of $\psl$ as an embedding into $\mathsf{PGL}_2(\mathbf{C}(q))$}
Whenever convenient, elements of projective groups will described as linear fractional maps or by projective matrices.
Define the three involutions in $\pgl$
$$
U:=x\mapsto 1/x, \quad V:=x\mapsto -x, \quad K:=x\mapsto 1-x, 
$$
and define the three elements in $\psl$ by
$$
L:=KU: x\mapsto 1-1/x, \quad 
T:=KV: x\mapsto 1+x, \quad S:=UV: x\mapsto -1/x.
$$
The following presentations are well known~\cite{zeytin2}:
\begin{align*}
\pgl&=\langle U, V, K \, |\, U^2=V^2=K^2=(UV)^2=(KU)^3=1 \rangle,\\
&=\langle U, T \, |\, U^2=(UTU^{-2})^2=(UTUT^{-1})^3=1 \rangle,\\
\psl&=\langle S, L \, |\, S^2=L^3=1 \rangle,\\
&=\langle S, T \, |\, S^2=(TS)^3=1 \rangle.
\end{align*}
Observe that
\begin{align*}
\mathsf{PGL}_2(\C(q))=\mathsf{PGL}_2(\C[q]):=
\Bigl\{
&M:[m:n] \in \mathsf{P^1}(\C[q]) \mapsto [Am+Bn: Cm+Dn]\in \mathsf{P^1}(\C[q]) \, \big|\\ 
&\,A,B,C,D \in \C[q], \quad AD-BC\neq 0\Bigr\}.
\end{align*}
Let $(\Psi, \psi)$ be a pair with
	\begin{align*}
			\Psi: \psl &\to \mathsf{PGL}_2(\C(q))\\
		\psi: {\mathsf P^1}(\Z) &\to {\mathsf P^1}(\C[q]),
	\end{align*}
satisfying the equivariance and the quantization conditions:
\begin{alignat*}{2}
    \psi(Mx)&=\Psi(M)(\psi(x)),\qquad  && \forall M\in \psl, \, \forall  x\in  {\mathsf P^1}(\Z); \\
    \psi(1+m) &= 1+q\psi(m),     && \forall  m\in {\mathsf P^1}(\Z).
\end{alignat*}
Denote 
$$
\Psi(T)=:\mathcal T, \quad \Psi(S)=:\mathcal S, \quad \Psi(L)=:\mathcal L, \, \mbox{etc}.
$$ 
We observe that for any $m\in\Z$,
$$
\mathcal{T}(\psi(m))=\Psi(T)(m)=\psi(T(m))=\psi(1+m)=1+q\psi(m).
$$
Therefore $\mathcal{T}(x)=1+qx$. 
In order to determine $\Psi$, we are now looking for $\mathcal{S}=\Psi(S)$ such that 
$(\mathcal{T}\mathcal{S})^3=1$.

Let $\mathcal{T}$ be the projective matrix
$$
\left[\begin{matrix}
	q
	&
	1
	\\ 
	0
	& 
	1
\end{matrix}\right].
$$
\begin{theorem}
There exist exactly three representations $\Psi: \psl \to \mathsf{PGL}_2(\C(q))$ with $\Psi(T)=\mathcal{T}$:
\begin{itemize}
\item Morier-Genoud and Ovsienko's representation 
$\Psi: \psl \to \mathsf{PGL}_2(\Z[q,1/q])$ defined by
$$
\Psi(S)=\mathcal{S}=\left[\begin{matrix}
	0
	&
	-1
	\\ 
	q
	& 
	0
\end{matrix}\right],
$$
with an extension to $\pgl$ defined by
$$
\Psi(V)=\mathcal{V}=\left[\begin{matrix}
	q
	&
	1-q
	\\ 
	q-q^2
	& 
	-q
\end{matrix}\right].
$$
\item A pair of conjugate representations $\Psi^\pm: \psl \to \mathsf{PGL}_2(\Z[\omega][q,1/q])$ defined by
$$
\Psi^\pm(S)=\mathcal{S}^\pm=
	\left[\begin{matrix} 
		1
		&
		q^{-1}
		\\ 
		-q+\omega^{\pm 1}
		& 
		-1
	\end{matrix}\right], \quad \omega=\exp\left(\frac{2\pi i}{6}\right),\\
$$
with an extension to $\pgl$ defined by
$$
\Psi^\pm(V)=\mathcal{V}^\pm=
\left[\begin{matrix}
	1 & 
	\frac{1+q^{-1}}{q-\omega^{\pm 1}}
	\\ 
	1-q & 
	-1
\end{matrix}\right].
$$
\end{itemize}
\end{theorem}
\begin{proof}
Suppose
$$ 
\Psi(S)=:\mathcal{S}=\frac{Ax+B}{Cx+D} \quad  \bigl(A,B,C,D \in \C[q] \bigr).
$$
The modular group relations $S^2=(TS)^3= 1$ forces
$
\mathcal{S}^2=1,\ (\mathcal{T}\mathcal{S})^3=1.
$
The first relation gives us $A^2=D^2$ and $A+D=0$ or $B=0=C$. If $A=D\neq 0$, we get $B=C=0$ and obtain $\mathcal{S}=I$, which violates the relation $(\mathcal{T}\mathcal{S})^3=1$.
This leaves us two cases:
\begin{itemize}
\item[\textbf{Case I:}]  If $A=D=0$, we can assume that $B=1/C$ where $C$ is an expression such that $C^2\in \mathbb{C}[q]$. Then (changing to matrix notation for convenience)
$$ \mathcal{T}\mathcal{S}
	=
	\left[\begin{matrix}
		q & 1 \\ 0 & 1
	\end{matrix}\right]
	\left[\begin{matrix}
		0 & 1/C \\ C & 0
	\end{matrix}\right]
	=
	\left[\begin{matrix}
		C & q/C \\ C & 0
	\end{matrix}\right]
$$
and thus the second relation becomes
$$
	1= 
	\left[\begin{matrix}
		C & q/C \\ C & 0
	\end{matrix}\right]^3
	=
	\left[\begin{matrix}
		C^3+2qC & qC+q^2/C \\ C^3+qC & qC
	\end{matrix}\right]
$$
which has the only solution $C^2=-q$. This yields the matrix
$$ 
\mathcal{S}=	
\left[\begin{matrix}
	0 & \pm iq^{-1/2} \\ \mp iq^{1/2} & 0
\end{matrix}\right]=	
\left[\begin{matrix}
0 & q^{-1/2} \\ -q^{1/2} & 0
\end{matrix}\right]=	
\left[\begin{matrix}
0 & 1 \\ -q & 0
\end{matrix}\right],
$$
i.e. $\mathcal{S}(x)=-1/(qx)$. We conclude that
$$
\mathcal{S}:=x\mapsto-\frac{1}{qx}, \quad \mathcal{T}:=x\mapsto 1+qx.
$$
This defines the representation $\Psi$ on $\psl$. 
Note that
$$
\Psi(L)=\Psi(TS)=\mathcal T \mathcal S=\mathcal L, 
$$
where $\mathcal L :=x\mapsto1-1/x$ (there is no $q$ involved).
To extend $\Psi$ to $\pgl$, recall that $V(x)=-x$ and let (using projective matrices for convenience)
$$
\Psi(V)=\mathcal{V}=
\left[\begin{matrix}
	A  & B  \\ C & D 
\end{matrix}\right] \quad  \bigl(A,B,C,D \in \C[q] \bigr).
$$
We have the relations
\begin{align}
	\mathcal{V}^2 &=1, \label{eq:rel1}\\
	(\mathcal{T}\mathcal{V})^2 &=1, \label{eq:rel2}\\
	(\mathcal{S}\mathcal{V})^2 &=1.	\label{eq:rel3}
\end{align}
From (\ref{eq:rel1}) we get $A ^2=D ^2$ and, $A +D =0$ or $B =0= C $. The case where $A =D \neq 0$ gives the identity, which contradicts with $\Psi$ being an injection. Like above, for the case $A =D =0$, we may assume that $B = 1/C$. On the other hand, (\ref{eq:rel2}) gives
$$
I =
\left[\begin{matrix}
	 C 
	&
	q/C
	\\ 
	 C 
	& 
	0
\end{matrix}\right]^2
=
\left[\begin{matrix}
	 C ^2+q
	&
	q
	\\ 
	 C ^2
	& 
	q
\end{matrix}\right],
$$
which has no solution. For the last case where $A =-D \neq0$, we may assume that $A =1$ and $D =-1$. Then
$$
\mathcal{T}\mathcal{V}
=
\left[\begin{matrix}
	q+ C 
	&
	qB -1
	\\ 
	 C 
	& 
	-1
\end{matrix}\right].
$$
We should have $\mathcal{T}\mathcal{V}\not=1$ and $(\mathcal{T}\mathcal{V})^2=1$, so trace should be zero (since an element $M\in \mathsf{PGL}(\C(q))$ is involutive if and only if its trace is 0). Hence, we get $ C =1-q$ and obtain
$$
\mathcal{V}
=
\left[\begin{matrix}
	1
	&
	B 
	\\ 
	1-q
	& 
	-1
\end{matrix}\right].
$$
Then
$$
\mathcal{S}\mathcal{V}
=
\left[\begin{matrix}
	1-q 	&	 -1 \\
	-q		&	-qB 
\end{matrix}\right]
$$
which should have zero trace, by Equation \ref{eq:rel3}. Thus, $B =\frac{1-q}{q}$. We conclude that
$$
\mathcal{V}
=
\left[\begin{matrix}
	1
	&
	\frac{1-q}{q}
	\\ 
	1-q
	& 
	-1
\end{matrix}\right]
=
\left[\begin{matrix}
	q
	&
	1-q
	\\ 
	q-q^2
	& 
	-q
\end{matrix}\right].
$$
This defines the representation $\Psi$ on $\psl$. 
Note that
$$
\mathcal U=\Psi(U)=
\left[\begin{matrix}
	q-1
	&
	1
	\\ 
	q
	& 
	1-q
\end{matrix}\right], \quad
\mathcal K=\Psi(K)=
\left[\begin{matrix}
	1
	&
	-q
	\\ 
	1-q
	& 
	-1
\end{matrix}\right].
$$
\item[\textbf{Case II:}] 
	Let $D=-A\neq 0$. Then we can assume that $A=1,D=-1$. Thus we have 
$$ \mathcal{S}=	
	\left[\begin{matrix}
		1 & B \\ C & -1
	\end{matrix}\right]
	,\
	\mathcal{T}\mathcal{S}=
	\left[\begin{matrix} q+C & qB-1 \\ C & -1
	\end{matrix}\right]
.$$
Direct computation yields
\begin{align*}
	(\mathcal{T}\mathcal{S})^3
	&=
		\left[\begin{matrix} 
			(q+C)^3+C(qB-1)[2(q+C)-1]
			&
			(qB-1)[(q+C)^2+C(qB-1)-(q+C)+1] 
			\\ 
			(q+C)^2+C(qB-1)-(q+C)+1
			& 
			[(q+C)-1](qB-1)-C(qB-1)-1
		\end{matrix}\right]\\
	&=
		\left[\begin{matrix} 
		u^3+v(2u-1)
		&
		v(u^2+v-u+1)
		\\ 
		u^2+v-u+1
		& 
		(u-1)v-v-1
		\end{matrix}\right]=1.
\end{align*}
by the substitution $u=q+C\ ,v=C(qB-1)$. This yields the  system
\begin{align*}
	u^3+v(2u-1)  &=(u-1)v-v-1\\
	v(u^2+v-u+1) &=0 \\
	u^2-u+1+v &=0 
\end{align*}
which reduces to $u^2-u+1 =0$ and $v =0$. Hence, $u=\omega$ or $\overline{\omega}=\omega^{-1}$ where $\omega=\frac{1}{2}+ i \frac{\sqrt 3}{2}$ is a 6th primitive root of unity.
Thus we obtain $C=-q+\omega^{\pm 1}$. Besides, $0=v=qB-1$ yields $B=q^{-1}$. We conclude that
$$\mathcal{S^\pm}=
	\left[\begin{matrix} 
		1
		&
		q^{-1}
		\\ 
		-q+\omega^{\pm 1}
		& 
		-1
	\end{matrix}\right], \quad \mathcal{T}=\left[\begin{matrix} 
		q
		&
		1
		\\ 
		0
		& 
		1
	\end{matrix}\right].
$$
This defines the representation $\Psi^\pm$  on $\psl$.
To extend $\Psi^\pm$ to $\pgl$ let, as in Case I,
$$\mathcal{V}=
\left[\begin{matrix}
	A  & B  \\  C  & D 
\end{matrix}\right].
$$
By (\ref{eq:rel1}) and (\ref{eq:rel2}), the only interesting case is $D =-A \neq 0 $. 
As before, assume $A =1, D =-1$. Then (\ref{eq:rel2}) yields
$$
	1= (\mathcal{T}\mathcal{N})^2=
	\left(
	\left[\begin{matrix}
		q & 1 \\ 0 & 1
	\end{matrix}\right]
	\left[\begin{matrix}
		1 & B  \\  C  & -1
	\end{matrix}\right]
	\right)^2=
	\left[\begin{matrix}
		q+ C  & 
		qB -1 
		\\ 
		 C  & 
		-1
	\end{matrix}\right]^2.
$$
This implies that $ C =1-q$, since the trace should be zero. Also, by (\ref{eq:rel3}):
$$
\mathcal{V}\mathcal{S}^{\pm}=
\left[\begin{matrix}
	1 & B \\ 1-q & -1
\end{matrix}\right]
\left[\begin{matrix}
	1 & q^{-1} \\ -q+\omega^{\pm1} & -1
\end{matrix}\right]=
\left[\begin{matrix}
	1+B(\omega^{\pm1}-q) &
	q^{-1}-B 
	
	\\ 
	1-\omega^{\pm1} & 
	q^{-1}
\end{matrix}\right].
$$
should have zero trace. Therefore, $B=\frac{1+q^{-1}}{q-\omega^{\pm1}}$. As a result, we have
$$
\mathcal{V}^{\pm}=
\left[\begin{matrix}
	1 & 
	\frac{1+q^{-1}}{q-\omega^{\pm1}}
	\\ 
	1-q & 
	-1
\end{matrix}\right].
$$
Note that 
$$
\mathcal L^\pm=\Psi^\pm(L)=
\left[\begin{matrix}
	\omega^{\pm1}
	&
	0
	\\ 
	-q+\omega^{\pm1}
	& 
	-1
\end{matrix}\right], \quad 
\mathcal U^\pm =\Psi^\pm(U)=
\left[\begin{matrix}
	q-\omega &
	\omega+1
	\\ 
	q^2\omega^2-q & 
	-q+\omega
	\end{matrix}\right],
	$$
	$$
\mathcal K^\pm=\Psi^\pm(K)=
\left[\begin{matrix}
	1
	&
	-1+\frac{q(1+1/q)}{q-\omega^{\pm 1}}
	\\ 
	1-q
	& 
	-1
\end{matrix}\right].
$$
\end{itemize}
\end{proof}

\begin{remark}
Note that the representation $\Psi$ remains well-defined on $\pgl$ when we specialize to any non-zero value of $q\in \C$. 
The representation $\Psi^\pm$ has one exception to this rule: if $q=\pm \omega$, then $\mathcal K^\pm,\mathcal V^\pm$ become singular. Hence the representation does not extend to $\pgl$ in this case. The representations $\Psi^\pm$ are still well-defined on $\psl$.
\end{remark}

{\small
\begin{table}[htp]
\caption{Three quantization representations ($\mbox{pdet}$ denotes the projective determinant, which is well-defined up to multiplication by a square within the ring in context).}
\begin{center}
\begin{tabular}{|c|c|c|c|c|}
\hline
&$\Psi$ &pdet&$\Psi^\pm$&pdet\\ \hline &&&&\\[-.6cm]

$\Psi(T)$
    &$\left[\begin{matrix} 
		q
		&
		1
		\\ 
		0
		& 
		1
	\end{matrix}\right]$&$q$&$\left[\begin{matrix} 
		q
		&
		1
		\\ 
		0
		& 
		1
	\end{matrix}\right]$&$q$\\ [.6cm]  \hline &&&&\\[-.6cm]
$\Psi(S)$ & $\left[\begin{matrix} 
		0
		&
		-1
		\\ 
		q
		& 
		0
	\end{matrix}\right]$ &$q$&$\left[\begin{matrix} 
		1
		&
		q^{-1}
		\\ 
		-q+\omega^{\pm 1}
		& 
		-1
	\end{matrix}\right]$&$  q$\\ [.6cm]  \hline &&&&\\[-.6cm]
$\Psi(L)$&$\left[\begin{matrix} 
		1
		&
		-1
		\\ 
		1
		& 
		0
	\end{matrix}\right]$&$1$&$\left[\begin{matrix}
	\omega^{\pm1}
	&
	0
	\\ 
	-q+\omega^{\pm1}
	& 
	-1
\end{matrix}\right]$&$ 1 $\\ [.6cm]  \hline &&&&\\[-.6cm]
$\Psi(U)$&$\left[\begin{matrix}
	q-1
	&
	1
	\\ 
	q
	& 
	1-q
\end{matrix}\right]$&$-q^2+q-1$&$\left[\begin{matrix}
	q-\omega &
	\omega+1
	\\ 
	q^2\omega^2-q & 
	-q+\omega
	\end{matrix}\right]$&$ q^2-q+1 $\\ [.6cm]  \hline &&&&\\[-.6cm]
$\Psi(V)$&$\left[\begin{matrix}
	q
	&
	1-q
	\\ 
	q-q^2
	& 
	-q
\end{matrix}\right]$&$-q^2+q-1$&$\left[\begin{matrix}
	1 & 
	\frac{1+q^{-1}}{q-\omega^{\pm1}}
	\\ 
	1-q & 
	-1
\end{matrix}\right]$&$q^2-q+1$\\ [.6cm]  \hline &&&&\\[-.6cm]
$\Psi(K)$&$\left[\begin{matrix}
	1
	&
	-q
	\\ 
	1-q
	& 
	-1
\end{matrix}\right]$&$-q^2+q-1$&$\left[\begin{matrix}
	1
	&
	-1+\frac{q(1+1/q)}{q-\omega^{\pm 1}}
	\\ 
	1-q
	& 
	-1
\end{matrix}\right]$&$q^2-q+1$\\ [.6cm]  \hline
\end{tabular}
\end{center}
\label{default}
\end{table}%
}

Having established the existence of the quantization maps, 
we adopt the following notations for their images:
\begin{align*}
{\mathsf{PSL}}_2(\Z,q):=      &\,\, \Psi(\psl),        & {\mathsf{PGL}}_2(\Z,q):=       &\,\, \Psi(\pgl), \\
{\mathsf{PSL}}_2^\pm (\Z,q):=&\,\, \Psi^\pm(\psl), &{\mathsf{PGL}}_2^\pm (\Z,q):=&\,\, \Psi^\pm(\pgl).
\end{align*}
Note that 
\begin{align*}
    {\mathsf{PSL}}_2(\Z,q) 
    &<\,  
    \left\{
        \left[
            \begin{matrix}
            	A
            	&
            	B
            	\\ 
            	C
            	& 
            	D
            \end{matrix}
        \right]
        | \,A,B,C,D \in \Z[q], \,  AD-BC\in  q^\Z 
    \right\}    \\
    &<{\mathsf{PGL}}_2(\Z[q,q^{-1}]):=  \,
\left\{
\left[\begin{matrix}
	A
	&
	B
	\\ 
	C
	& 
	D
\end{matrix}\right]
| \,A,B,C,D \in \Z[q], \,  AD-BC\in  \pm q^\Z \right\}
\end{align*}
\begin{align*}
{\mathsf{PGL}}_2(\Z,q) &<\, \left\{
\left[\begin{matrix}
	A
	&
	B
	\\ 
	C
	& 
	D
\end{matrix}\right]
| \,A,B,C,D \in \Z[q], \,  AD-BC\in  (-q^2+q-1)^\Z q^\Z \right\},\\
&<{\mathsf{PGL}}_2(\Z[q,q^{-1},(-q^2+q-1)^{-1}])\\
 &:=\left\{
\left[\begin{matrix}
	A
	&
	B
	\\ 
	C
	& 
	D
\end{matrix}\right]
| \,A,B,C,D \in \Z[q], \,  AD-BC\in \pm (-q^2+q-1)^\Z q^\Z \right\},
\end{align*}
\begin{align*}
{\mathsf{PSL}}_2^\pm (\Z,q)<&\, {\mathsf{PGL}}_2(\Z[\omega][q,q^{-1}]):=\\ 
&\left\{
\left[\begin{matrix}
	A
	&
	B
	\\ 
	C
	& 
	D
\end{matrix}\right]
| \,A,B,C,D \in \Z[\omega][q], \,  AD-BC\in \omega^\Z q^\Z \right\}, \\
{\mathsf{PGL}}_2^\pm (\Z,q)<&{\mathsf{PGL}}_2(\Z[\omega][q,q^{-1},(-q^2+q-1)^{-1}]):=\\ 
&\left\{
\left[\begin{matrix}
	A
	&
	B
	\\ 
	C
	& 
	D
\end{matrix}\right]
| \,A,B,C,D \in \Z[\omega][q], \,  AD-BC\in \omega^\Z (-q^2+q-1)^\Z q^\Z  \right\}.\\
\end{align*}
Also note that we have the natural inclusions of the modular groups
\begin{align*}
\psl < &\,\,\pgl<{\mathsf{PGL}}_2(\Z[q,q^{-1}])<{\mathsf{PGL}}_2(\Z[q,q^{-1},(-q^2+q-1)^{-1}])\\
\psl < &\,\,\pgl<{\mathsf{PGL}}_2(\Z[\omega][q,q^{-1}])<{\mathsf{PGL}}_2(\Z[\omega][q,q^{-1},(q^2-q+1)^{-1}]),
\end{align*}
induced by the inclusions 
$\Z\hookrightarrow \Z[q,1/q]\hookrightarrow \Z[q,1/q, 1/(-q^2+q-1)]$ and 
$\Z\hookrightarrow \Z[\omega][q,1/q]\hookrightarrow \Z[\omega][q,1/q, 1/(q^2-q+1)]$.

The group ${\mathsf{PSL}} (\Z,q)<{\mathsf{PGL}}_2(\Z[q,q^{-1}])$ etc. are rather small subgroups of the right hand sides. For example, by results of \cite{jouteur} we know that the traces of elements of ${\mathsf{PSL}} (\Z,q)$ are palindromic polynomials up to a signed power of $q$.

The representation $\Psi$ is faithful because it specializes to the identity on $\pgl$ at $q=1$. 
The proof of the proposition below is a routine check:
\begin{proposition}\label{pmfaithful}
The representation $\Psi^\pm$ is faithful.
In fact, at $q=1$ the representation $\Psi^\pm$ is conjugate to the subgroup
$
\pgl < {\mathsf{PGL}}_2(\Z[\omega][q,q^{-1},(q^2-q+1)^{-1}])
$
via the transformation $H:=x+\omega^\pm$; in the sense that
\begin{align*}
H T_1 H^{-1}=T, \quad
H S_1^\pm H^{-1}=S,\quad
H  L_1^\pm H^{-1}=L,\\
H  U_1^\pm H^{-1}=U,\quad
H  V_1^\pm H^{-1}=V,\quad
H  K_1^\pm H^{-1}=K.
\end{align*}
where $T_1=\mathcal{T}|_{q=1}$ etc.
\end{proposition}

\section{Quantizations of  rationals}
Having established three quantization representations $\Psi$, $ \Psi^\pm$, now we ask: does there exists equivariant functions $\psi, \psi^\pm: {\mathsf P^1}(\Z) \to {\mathsf P^1}(\C[q])$ with respect to these representations?

\subsection{Morier Genoud and Ovsienko's representation $\Psi$} 
In the case of $\psi$, the equivariance conditions for $\psl$ reads as (using temporarily the notation $[x]_q$ for $\psi(q)$):
\begin{align}
[1+x]_q=&1+q[x]_q \quad (\mbox{$T$-equivariance}), \notag \\
\left[-\frac{1}{x}\right]_q=&-\frac{1}{q[x]_q} \quad (\mbox{$S$-equivariance}), \label{MGOequivariance}\\
\left[1-\frac{1}{x}\right]_q=&1-\frac{1}{[x]_q} \quad (\mbox{$L$-equivariance}) \notag
\end{align}
(in fact any pair of equations above is sufficient since any two of $S, L, T$ generate $\psl$).
For the extension of $\Psi$ to $\pgl$, these conditions become (beware these conditions are inconsistent over $\Q$
as is explained further below)
\begin{align}
\left[-x\right]_q=&\frac{q[x]_q+(1-q)}{q(1-q)[x]_q-q} \quad (\mbox{$V$-equivariance}) &\notag\\
\left[\frac{1}{x}\right]_q=&\frac{(q-1)[x]_q+1}{q[x]_q+1-q} \quad (\mbox{$U$-equivariance}) &\\
[1-x]_q=&\frac{[x]_q-q}{(1-q)[x]_q-1} \quad (\mbox{$K$-equivariance}) & \notag
\end{align}
(in fact only the $U$- and $T$-equivarience are sufficient, since they generate $\pgl$).
Now since we require $[1]_q=1$ for quantization, setting $x=1$ in the equivariance condition for $U$ gives
$$
1=\frac{(q-1)+1}{q+1-q}=q,\\
$$
which is inconsistent (if we require $[1]_q=q$ then we get $q=\mathcal U(q)=1$).
In fact, $[1]_q$ must be one of the two fixed points of $\mathcal U$, i.e. 
$$
x=\frac{(q-1)\pm \sqrt{q^2-q+1}}{q}.
$$
This shows that there is no consistent way to define $\psi$ in a $\pgl$-equivariant way on ${\mathsf P^1}(\Z)$
(unless we extend the target space and define $\psi(1)$ accordingly, sacrificing the quantization condition $\psi(1)=1$
or $\psi(1)=q$). Note that the natural extension of  $\psi$ to $\Real\setminus \Q$ is $\pgl$-equivariant~\cite{jouteur}.

Is it possible to consistently define $\psi$ in a $\psl$-equivariant way, as given by Morier Genoud and Ovsienko's?
The answer is known to be yes~\cite{genoud}. We will reprove this result since we want to do the same for the representations $\Psi^\pm$. In order to do this, we need to return to our initial setting of Equation~\ref{setting}.

If an equivariant pair $(\Psi, \psi)$  exists, and if the $G$-action on $X$ is transitive, $\psi$ is determined by its value $\psi(x_0):=y_0$ on any point $x_0\in X$. 
Indeed, assume $x\in X$.  
By transitivity,  there is a $g\in G$ with $gx_0=x$. 
Hence $\psi(x)=\psi(gx_0)=\Psi(g)\psi(x_0)=\Psi(g)y_0$.

On the other hand, there may exist other elements $h$ with $hx_0=x$; 
equivalently $k:=h^{-1}g \in G_{x_0}$, the stabilizer of $x$ under $G$.
This forces $\Psi(k)\in G_{y_0}$. 
Hence we have the following necessary condition on the $G$-sets $X$ and $Y$ for the existence of an equivariant pair:
$$
\stab_G({x})<\stab_G({\psi(x)}) \quad \forall x\in X 
$$
\begin{lem}\label{lem:stabilizer}
Suppose that $X$ is a transitive $G$-set and $\Psi:G \to  \mathrm{Aut}(Y)$ a homomorphism. 
Let $x_0\in X, y_0 \in Y$. Then there exists a map $\psi: X \to Y$ so that $\psi(x_0)=y_0$ and $(\Psi, \psi)$ is an equivariant pair if and only if $\Psi(\stab_G(x_0))\subseteq\stab_{\mathrm{Aut}(Y)}(y_0)$. 
If such a function $\psi$ exists, then it is unique. 

\end{lem}
\begin{proof}
	Let $x\in X$. Choose $g\in G$ so that $x=g\cdot x_0$. Define $\psi(x):=\Psi(g)\cdot y_0\in Y$. Then, $\psi:X \to Y$ is a well-defined function if and only if $\Psi(g)$ stabilizes for $y_0$ whenever $g$ stabilizes $x_0$.
\end{proof}

\begin{proposition}(\cite{Valentin}) \label{thm:MGOconsistent}
Equivariance equations (\ref{MGOequivariance}) are consistent for the $\psl$-action; i.e. there exists functions $\psi$ satisfying them.
\end{proposition}

\begin{proof}
Let  $x_0=1$. The stabilizer for $n=1$ for the $\psl$-action on $\mathsf P^1(\Z)$ is
$$ \{TST^nST^{-1}:n \in \mathsf Z\}=\langle TSTST^{-1}\rangle,$$
where
$$
A:=TSTST^{-1}=\begin{bmatrix} 
        		0	&	1 \\ 
        		-1	&	2
	           \end{bmatrix},
    \text{ with \ }
\mathcal{A}:=\Psi(A)=\begin{bmatrix} 
		0
		&
		q
		\\ 
		-1
		& 
		q+1
	\end{bmatrix}
$$
By Lemmma~\ref{lem:stabilizer}, the condition for $\psi(1)$ is
$$ \Psi\left(\stab_{\psl}(1)\right)=\langle\mathcal{A}\rangle \subset \stab_{\pglc}(\psi(1)), $$
which is equivalent to $\psi(1)$ being a fixed point of $\mathcal{A}$. Solving $\mathcal{A}x=x$, we obtain $1$ and $q$ as fixed points of $\mathcal{A}$ and hence as possible choices for $\psi(1)$.
\end{proof}
Note that $\psi(1)=1\iff \psi(\infty)=\infty$, and $\psi(q)=1\iff \psi(\infty)=1/(1-q)$.
The corresponding quantization maps $\psi$ were respectively denoted $[x]_q^\sharp$ and $[x]_q^\flat$ in~\cite{jouteur}.

\subsection{The conjugate representations $\Psi^\pm$}
In this case the equivariance conditions for $\psl$ reads as (using temporarily the notation $[x]_q^\pm$ for $\psi^\pm(q)$):
\begin{align}\label{topkaraq}
[1+x]_q^\pm=&1+q[x]_q \quad (\mbox{$T$-equivariance}),\notag\\
\left[-\frac{1}{x}\right]_q^\pm=&\frac{q[x]_q+1}{q(-q+\omega^{\pm 1})[x]_q-q} \quad (\mbox{$S$-equivariance}),\\
\left[1-\frac{1}{x}\right]_q^\pm=&\frac{\omega^{\pm 1}[x]_q}{(-q+\omega^{\pm 1})[x]_q-1} \quad (\mbox{$L$-equivariance}).\notag
\end{align}
It follows that 
\begin{align}\label{topkarapgl}
\left[-x\right]_q^\pm=&\frac{q(q-\omega^{\pm 1})[x]_q+(1+q)}{q(1-q)(q-\omega^{\pm 1})[x]_q-q(q-\omega^{\pm 1})} \quad (\mbox{$V$-equivariance}),\notag\\
\left[\frac{1}{x}\right]_q^\pm=&\frac{(q-\omega^{\pm 1})[x]_q+1+\omega^{\pm 1}}{q(-1+\omega^{\pm1})(q-\omega^{\pm1})[x]_q-q+\omega^{\pm 1}} \quad (\mbox{$U$-equivariance }),\\
[1-x]_q^\pm=&\frac{(q-\omega^{\pm1})[x]^{\pm}_q+1+\omega^{\pm1}}{(q-\omega^{\pm1})(1-q)[x]_q^\pm+(-q+\omega^{\pm1})} \notag
 \quad (\mbox{$K$-equivariance }).
\end{align}
In this case, the two fixed points of $\mathcal U^\pm$ are
$$
\frac{q-\omega^{\pm1}\pm \omega^{\pm1}\sqrt{q^{2}-q+1}}{q\left(1+\omega^{\pm2} q\right)},
$$
and the equations can not be made consistent over $\C[q,q^{-1}, (q^2-q+1)^{-1}]$. 
Hence, no $\Psi^\pm$-equivariant functions $\psi^\pm$ exists on $\pgl$. As for the group $\psl$ we have
\begin{proposition}
Equivariance equations (\ref{topkaraq}) are consistent for the $\psl$-action; i.e. there exists functions $\psi$ satisfying them.
\end{proposition}

\begin{proof}
To determine possible choices for $\psi(1)$, we again consider the stabilizer of $x_0=1$ for the action of $\psl$ on $\Z$. Recall that 
 $\stab_{\psl}(1)=\langle A \rangle$ where $A=TSTST^{-1}$. Thus
\begin{align*}
	\Psi^\pm(A)	&=:\mathcal{A^\pm}=\Psi(TSTST^{-1}) \\
			&=\mathcal{TS^\pm TS^\pm} \mathcal{T}^{-1}=(\mathcal{TS^\pm})^2\mathcal{T}^{-1} \\
			&=	\begin{bmatrix}
					\omega^{\pm 1} & 0 \\ -q+\omega^{\pm 1} & -1
				\end{bmatrix}^2
				\begin{bmatrix}
					1 & -1 \\ 0 & q
				\end{bmatrix} 
			=\begin{bmatrix}
				\omega^{\pm 2} & 0 \\ (q-\omega^{\pm 1})(1-\omega^{\pm 1}) & 1
				\end{bmatrix}
				\begin{bmatrix}
				1 & -1 \\ 0 & q
				\end{bmatrix} \\
			&=	\begin{bmatrix}
					\omega^{\pm 2} & -\omega^{\pm 2} \\ (q-\omega^{\pm 1})(1-\omega^{\pm 1}) & \omega^{\pm 1}-\omega^{\pm 2}+\omega^{\pm 1}q
				\end{bmatrix}
            =	\begin{bmatrix}
					\omega^{\pm 2} & -\omega^{\pm 2} \\ 
                    -\omega^{\pm 2}(q-\omega^{\pm 1}) & 1+\omega^{\pm 1}q
				\end{bmatrix}\\
            &=	\begin{bmatrix}
                -1 & 1 \\ 
                (q-\omega^{\pm 1}) & \omega^{\pm 1}+\omega^{\pm 2}q
            \end{bmatrix}
\end{align*}
As in the proof of Theorem~\ref{thm:MGOconsistent}, the condition for $\psi(1)$ is equivalent to $\psi(1)$ being a fixed point of $\mathcal{A}^{\pm}$. Solving $\mathcal{A}^{\pm}x=x$; we obtain $\omega^{-1}$ and $\frac{1}{1+\omega^2q}$ as fixed points of $\mathcal{A}^{+}$, $\omega$ and $\frac{1}{1-\omega q}$ as fixed points of $\mathcal{A}^{-}$ hence as possible choices for $\psi(1)$.
\end{proof}

%

\section{Specializations}
By {\it specialization} we mean fixing a value of  $q$ for the equivariant pair $(\Psi, \psi)$, or 
$(\Psi,^\pm \psi^\pm)$. Let $r \in \C$ and for $r \neq 0$ define the Möbius transformations
$$
T_r:= x\mapsto 1+r x, \quad S_r:= x\mapsto -\frac{1}{r x}.
$$
These generate a subgroup
$$
\mathsf{PSL}_2(\Z, q=r):= \langle T_r, S_r \rangle <\mathsf{PGL}_2(\C)  
$$
with a surjection (specialization map)
$$
\Psi_r: \mathsf{PSL}_2(\Z, q) \mapsto \mathsf{PSL}_2(\Z, q=r).
$$
(We may define $\mathsf{PSL}_2(\Z, q=0)$ to be the trivial group).
We can similarly define the group $\mathsf{PSL}^\pm_2(\Z, q)$, and for $q\neq 0, \omega^{\pm1}$ the groups 
$\mathsf{PGL}_2(\Z, q)$,
$ \mathsf{PGL}_2(\Z, q)$; along with the specialization map $\Psi^\pm_r$.
The transformations
$$
U_r, \, V_r, \, K_r, \, U_r^+, \, V_r^+, \, K_r^+,\,  U_r^-,\,  V_r^-,\,  K_r^-  \in \mathsf{PGL}_2(\C)
$$
are defined accordingly. We will use the notations $\psi_r$ and $\psi_r^\pm$ for the corresponding equivariant maps.

In particular we have
$$
\mathsf{PSL}_2(\Z, q=1) = \psl, \quad  
\mathsf{PGL}_2(\Z, q=1) = \pgl.
$$
By Proposition~\ref{pmfaithful}, we also have
$$
\mathsf{PSL}^\pm_2(\Z, q=1) \simeq \psl, \quad  
\mathsf{PGL}^\pm_2(\Z, q=1) \simeq \pgl.
$$
\begin{proposition}
If $r \in \C$ is not algebraic, then the specialization maps 
\begin{align*}
\Psi_r &:\mathsf{PGL}_2(\Z, q) \to \mathsf{PGL}_2(\Z, r)\\
\Psi^\pm_r&: \mathsf{PGL}^\pm_2(\Z, q) \to \mathsf{PGL}^\pm_2(\Z, r)
\end{align*}
are isomorphisms.
\end{proposition}
\begin{proof}
Let $r \in \C$ be transcendental and let $M\in \mathsf{PGL}_2(\Z, q)$. 
If $\Psi_r(M)$ is  identity, then the off-diagonal entries of $M$, which can be taken to be 
integral polynomials in $q$, 
must vanish at $q=r$. Hence $M$ must be the identity.
\end{proof}
\begin{proposition}\label{prop:cyclotome}
Let $r \in \C\backslash \{0\}$. 
Then $\Psi_r(\mathcal T^m)=I$ if and only if $r\neq 1$ is an $m$th root of unity.
Idem for $\Psi^\pm_r(\mathcal T^n).$
\end{proposition}
\begin{proof}
This is because, for $r \neq 1$,
$$
T_r^m(x)=r^mx+\frac{1-r^m}{1-r}=x  \quad (\forall x) \iff r^m=1.
$$
\end{proof}
The group
$\mathsf{PSL}_2(\Z,q=-1)\simeq \mathsf{PSL}_2^\pm(\Z,q=-1)$ is the symmetric group on three letters.
The group $\mathsf{PSL}_2\left(\Z,q=\exp{\frac{2\pi i}{k}}\right)$ is finite for $k<6$, and is solvable when $k=6$.

\medskip
The kernel of $\Psi_r$ may be non-trivial 
for some non-cyclotomic $r$, as the next example shows:

\begin{example} 
$$
\left(\mathcal T^3\mathcal S\right)^4 =
\frac{1}{1-q}
\left[\begin{matrix}
(1-q^5)(q^4+3q^3+3q^2+3q+1) & -q^2(1-q^3)(q^4+2q^3+q^2+2q+1) \\
(1-q^3)(q^4+2q^3+q^2+2q+1) & -q^2(1-q^4)(1+q)
\end{matrix}\right]
$$
In particular, $(\mathcal T^3\mathcal S)^4=1$ if and only if $q$ is a third root of unity or is one of \\
$r_{1,2}=0.2071067812 \pm 0.9783183435i$,\\
 $r_3=-0.5310100565$,\\
  $r_4=-1.883203506$. \\
Observe that $|r_{1,2}|=1$ and $r_2r_3=1$. However,  $r_{1,2}$ are not cyclotomic.
Further experiments indicate that if an element of $\mathsf{PSL}_2(\Z,q=r)$ collapses to identity at a real place $r$, then 
$r<0$. 
\end{example}

\begin{example} 
Let $X := (T_q^2 S_q T_q^3 S_q T_q^5 S_q  T_q^7 S_q)^5$.
Then 
\begin{align*}
P:=\gcd(X_{1,2}, X_{2,1})=
&(q^{4}+ q^{3}+ q^{2}+ q + 1)(q^{48}+ 11q^{47}+ 66q^{46}+ 286q^{45}+ 997q^{44}+  \\
&2960q^{43}+7743q^{42}+ 18246q^{41}+ 39342q^{40}+ 78517q^{39}+ 146316q^{38}+  \\
&256331q^{37}+424464q^{36}+ 667281q^{35}+ 999418q^{34}+ 1430283q^{33}+  \\
&1960540q^{32}+ 2579098q^{31}+3261413q^{30}+ 3969776q^{29}+ 4655997q^{28}+  \\
&5266354q^{27}+5748204q^{26}+ 6057177q^{25}+  6163639q^{24}+ 6057177q^{23}+   \\
&5748204q^{22}+5266354q^{21}+ 4655997q^{20}+ 3969776q^{19}+3261413q^{18}+ \\
& 2579098q^{17}+ 1960540q^{16}+ 1430283q^{15}+ 999418q^{14}+ 667281q^{13}+ \\
&424464q^{12}+ 256331q^{11}+ 146316q^{10}+ 78517q^{9}+ 39342q^{8}+ 18246q^{7}+\\
& 7743q^{6}+2960q^{5}+ 997q^{4}+ 286q^{3}+ 66q^{2}+ 11q+ 1)q^{20}
\end{align*}
\begin{figure}[H]
\centering
\fbox{\includegraphics[width=0.35\linewidth]{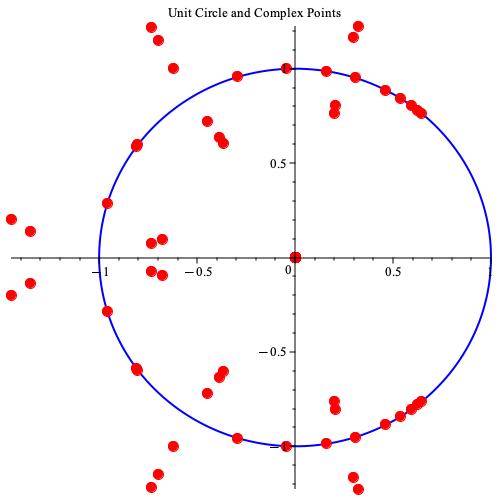}}
\caption{The locus where $X := (T_q^2 S_q T_q^3 S_q T_q^5 S_q  T_q^7 S_q)^5$ collapses to identity.}
\end{figure}
Observe  first the cyclotomic factor, which shows 
$\Psi_r(X)=I$ inside $\mathsf{PSL}_2(\Z,q=\exp{\frac{2 m \pi i}{5}})$, $m=1,2,3,4$. 
Also observe that the main factor is a palindromic polynomial. 
Hence, its roots are symmetric with respect to the circle. There are no real roots in this case.
For each root $r$, we have $\Psi_r(X)=I$.
The existence of many roots on the circle is somewhat surprising.
The corresponding element $X$ of $\mathsf{PSL}^\pm_2(\Z,q)$ yields identical results.
We don't know whether $\mathsf{PSL}_2(\Z,q=r)$ is a one-relator quotient of $\psl$,
where $r$ is a root of $P$. What we do know is that, by Proposition~\ref{prop:cyclotome}, these $\psl$-quotients are not finite
if $r^5\neq 1$.
Note in passing that the subgroup $\langle \Psi^{-1}(X)\rangle $ is represented by a modular graph~\cite{zeytin2} (a quotient graph of the Farey tree).
\end{example}

There are many questions pertaining to the groups $\mathsf{PSL}_2(\Z,q=r)$:
can one identify the loci
$$
\Lambda:=\left\{ r \, \middle| \Psi_r: \psl \to \mathsf{PSL}_2(\Z,q=r) \mbox{ is not injective } \right\} \subset \bar\Q?
$$ 
$$
\Lambda^\pm:=\left\{ r \, \middle| \Psi_r^\pm: \psl \to \mathsf{PSL}_2^\pm(\Z,q=r) \mbox{ is not injective } \right\} \subset \bar\Q?
$$ 
When this is the case, can one determine the kernel of $\Psi_r$?  Is the image of $\Psi_r$ always a 1-relator quotient of $\psl$?
We believe that $\mathsf{PSL}_2(\Z,q=r) \simeq \mathsf{PSL}^\pm_2(\Z,q=r) $ for all $r \in \C$.

Given an ideal $I\subset (\Z/N\Z)[q,q^{-1}]$ (e.g. $I=(P)$ where $P$ is the polynomial in Example 2), it is also of interest to study the kernels of the representations
$$
\psi: \psl \to {\mathsf{PGL}}_2((\Z/N\Z)[q,q^{-1}]/I),
$$
and the relations of these kernels to the principle congruence modular subgroups of $\psl$.

\section{Specialization to real values}
When $r \in \Real\backslash\{0\}$, the quantization map $\psi(x)=[x]_r$ is a real-valued function of $x$ and we can plot its graph. Table~\ref{tab:roots_figures} at the end of the paper contains the plots  $\psi$ for some positive values of $r$.
We observe the discontinuous though monotonic nature of these maps with jumps at rationals, as well as the fact that the plot converges to $y=x$ as $q\to 1$. 

Table~\ref{tab:roots_figures2} at the end of the paper depicts $\psi$ for some negative values of $r$.
We observe their discontinuous nature again, albeit qualitatively different from the case $r>0$. 
Our aim is now to elucidate this difference.

We first draw reader's attention to the resemblance of the plots in Table~\ref{tab:roots_figures2} with the plot below
(Figure~\ref{jimmm})  of the involution Jimm 
defined in \cite{conumerator}: 

\begin{figure}[H]
\centering
\fbox{\includegraphics[width=0.45\linewidth]{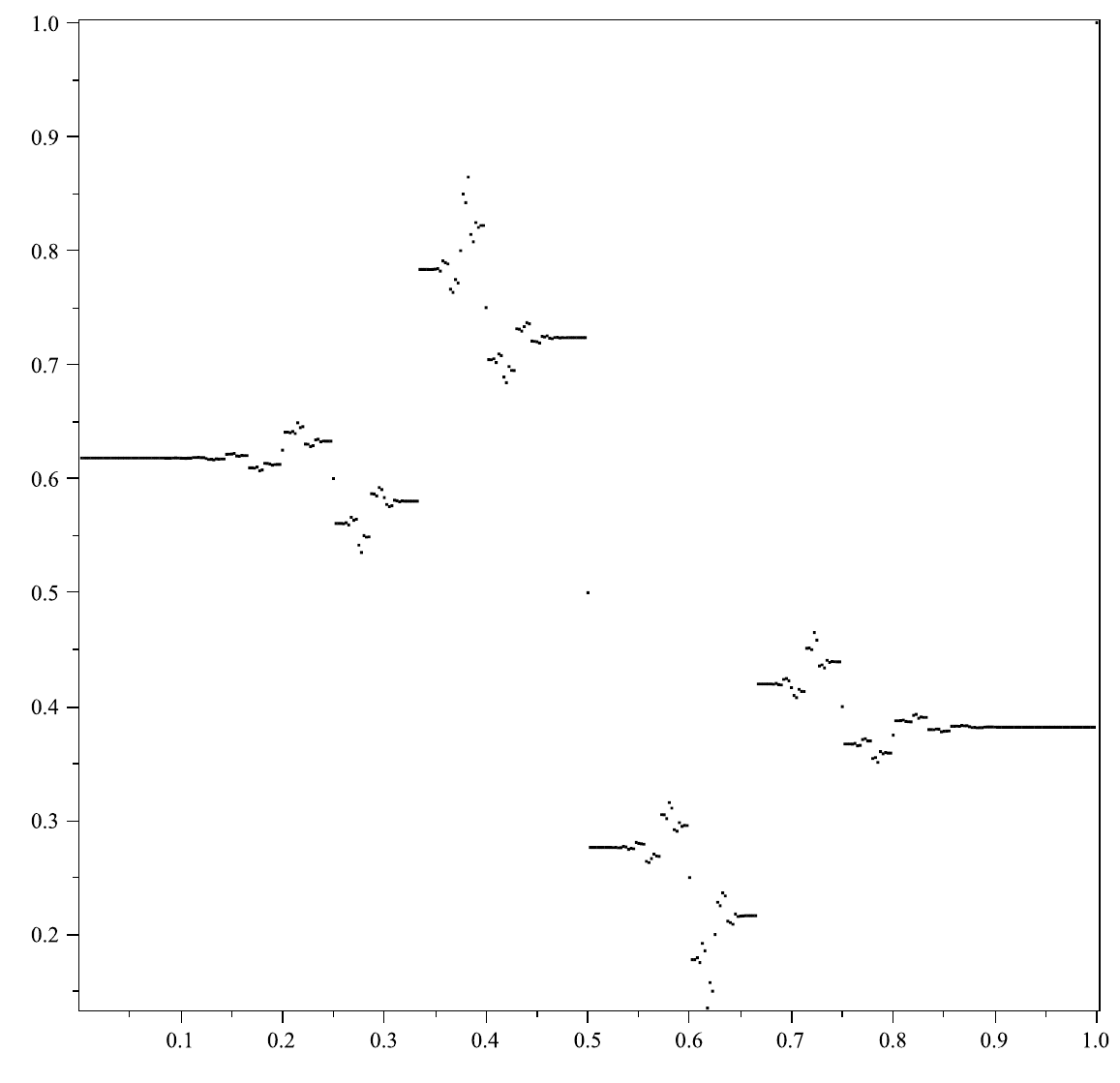}}
\label{jimmm}
\caption{Plot of the involution jimm on the unit interval.}\end{figure}

This involution $\Jimm$ is induced by Dyer's outer automorphism $\alpha$ of $\pgl$ as we explain below.
Dyer's outer automorphism is also manifested as an automorphism of the Farey tree (the two-sided Stern-Brocot tree)
of rationals \cite{uludagg}, which `{maximally violates}' the natural ordering of the nodes of the said tree, 
and also the natural ordering of its boundary. Hence, in a certain sense, $\Jimm$ is `{anti-monotonic}',
and the maps $\phi_r$ for $r\in \Real_{<0}$ visibly exhibits a similar behavior.

\subsection{Dyer's outer automorphism of $\pgl$ and the involution Jimm }
This  automorphism is defined in terms of the generators $U,V,K$ of $\pgl$ by
$$
\alpha(U)=U, \quad \alpha(K)=K, \quad \alpha(V)=UV \implies \alpha(T)=TU.
$$ 
It is easy to see that $\alpha$ is involutive, i.e. $\alpha\circ\alpha=Id$.
Since $\pgl$ is generated by $T$ and $U$, the set of equations $\alpha(U)=U$ and $\alpha(T)=TU$ is a complete set for determining $\alpha$.

By definition, a function $f$ is said to be {\it $\alpha$-equivariant} if the following system is satisfied:
\begin{equation}\label{eq:jimmfe}
f\left(\frac1x\right)=\frac{1}{f(x)}, \quad
f(1-x)=1-f(x), \quad 
f(-x)=-\frac{1}{f(x)} 
\implies
f(1+x)=1+\frac1{f(x)}
\end{equation}
Since $\pgl$ is generated by $T$ and $U$, the equations $f(1/x)=1/f(x)$ and $f(1+x)=1+1/f(x)$ are in fact sufficient for characterizing equivariance.

Now, the question is, do $\alpha$-equivariant functions $f$ exist?

Note that Equations~\ref{eq:jimmfe} are not consistent on $\mathsf P^1(\Z) $: setting $x=1$ in $f(1/x)=1/f(x)$ forces $f(1)=\pm 1$, 
and setting $x=0$ in $f(1-x)=1-f(x)$ forces $f(0) \in \{0,2\}$ whereas setting $x=0$ in $f(-x)=-1/f(x)$ implies
$f(0)^2=-1$. We see that the fixed points of $U$ and $V$ imposes an obstruction to the existence of an $\alpha$-equivariant function with respect to the $\pgl$-action on $\mathsf P^1(\Z) $. 

The index-2 subgroup $\psl<\pgl$ is not $\alpha$-invariant, since
$$
\alpha (\psl)=\alpha(\langle L, S\rangle)=\langle(\alpha(L), \alpha(S)\rangle=\langle L, V \rangle
$$
Therefore the functional equations for an $\alpha$-equivariant function on $\psl$ are
\begin{equation}\label{eq:jimmpsl}
f(1-1/x)=1-1/f(x), \quad 
f(-1/x)=-f(x).
\end{equation}
The largest  $\alpha$-invariant subgroup of $\psl$ is the index-2 subgroup $\Gamma<\psl$ generated by $\langle L, SLS \rangle$, since
$$
\alpha(L)=\alpha(KU)=KU=L,
\quad
\alpha(SLS)=VKUV=VU.KUKU.UV=SL^2S.
$$
Therefore $\alpha$ restricts to an outer automorphism of $\Gamma<\pgl$.
Note that $L.SLS=T^2\in \Gamma$.
\begin{lem}
The $\Gamma$-action on $\mathsf P^1(\Z)$ is transitive.
\end{lem}
\begin{proof}
Let $x\in \mathsf P^1(\Z)$. We want to find an $M\in \Gamma$ such that $Mx=\infty$.
Since the $\psl$-action on $\mathsf P^1(\Z)$ is transitive, there exists an $M\in \psl$ such that 
$Mx=0$. If $M\in \Gamma$, then $LM\in \Gamma$ too and $LMx=L0=\infty$.
If $M\notin\Gamma$, then $SM\in \Gamma$ and $SMx=S0=\infty$.
\end{proof}

The functional equations for an $\alpha$-equivariant function on $\Gamma$ are
\begin{equation}\label{eq:jimmsgfe}
f(1-1/x)=1-1/f(x), \quad 
f(-1/(1+x))=-1-1/f(x).
\end{equation}

\begin{theorem}\label{thm:jimmconsistent}
Systems (\ref{eq:jimmpsl})  and (\ref{eq:jimmsgfe}) are consistent on $\mathsf P^1(\Z) $; in fact there exists exactly 
two functions $f$ satisfying them, with
$$
f(1)=\frac{3+\sqrt{5}}{2}=\varphi^2 \mbox{ or } f(1)= \frac{3-\sqrt{5}}{2}=\bar\varphi^2,
$$
where $\varphi=\frac{1+\sqrt5}{2}$ is the golden ratio and $\bar\varphi=-\varphi^{-1}$ its Galois conjugate.
\end{theorem}
We denote the corresponding maps by $\Jimm_\sharp$ and $\Jimm_\flat$, so that
$\Jimm_\sharp(1)=\varphi^2$ and $\Jimm_\flat(1)=\bar\varphi^2$.
By transitivity of the $\Gamma$-action, these are defined on the whole set $\mathsf P^1(\Z)$.
\begin{proof}
It suffices to prove this for $\psl$, as  the proof  for $\Gamma$ leads to exactly the same result.
Let  $x_0=1$. Its stabilizer  for the action of $\psl$ on $\mathsf P^1(\Z)$ is
$$
 \{TST^nST^{-1}:n \in \mathsf Z\}=\langle TSTST^{-1}\rangle=\langle LTL^{-1}\rangle,
$$
where
$$
A:=TSTST^{-1}=\begin{bmatrix} 
        		-1	&	2 \\ 
        		-2	&	3
	           \end{bmatrix},
    \text{ with \ }
\alpha(A)=\begin{bmatrix} 
		-1
		&
		1
		\\ 
		-1
		& 
		2
	\end{bmatrix}
$$
By Lemmma~\ref{lem:stabilizer}, the condition for $f(1)$ is
$$ \Psi\left(\stab_{\psl}(1)\right)=\langle{\alpha(A)}\rangle \subset \stab_{\pglc}(f(1)), $$
which is equivalent to $f(1)$ being a fixed point of $\alpha(A)$. Solving $\alpha(A^2)x=x$, we obtain 
$\varphi^2$, $\bar\varphi^2$
as fixed points of $\alpha(A)$ and hence as possible choices for $f(1)$. 
\end{proof}

There is a little nuisance about the functions $\Jimm_\sharp$ and $\Jimm_\flat$ in that they don't land on the set 
$\mathsf P^1(\Z)$. (One would expect them to be involutions because $\alpha$ is involutive).
In fact, there does exist an involution $\Jimm$ of $\Q^+\subset \mathsf P^1(\Z)$  with $\Jimm(1)=1$ and satisfying  the equivariance equations 
$\Jimm(1/x)=1/\Jimm(x)$ and $\Jimm(1+x)=1+1/\Jimm(x)$. This function 
can then be extended to $\Q\setminus\{0\}$ via $f(-1)=-1/\Jimm(x)$, at the expense of sacrificing the equivariance conditions (\ref{eq:jimmpsl})  or (\ref{eq:jimmsgfe}), which the extended $\Jimm$ does not always obey (see \cite{conumerator}).
Moreover, for any irrational $x\in \Real$, the limit $\Jimm(y):=\lim_{y\to x} \Jimm(x)$ exists.
We extend $\Jimm$ to $\Real\setminus \Q$ as this limit and we keep the notation $\Jimm$ for the extended function.
It is continuous on $\Real\setminus \Q$, sending the set $\mathcal N$ of golden numbers 
(i.e. the $\pgl$-orbit of $\varphi$) to $\Q$ in a 2-1 manner. To wit,
$$
\Jimm(\Jimm_\sharp(x))=\Jimm(\Jimm_\flat(x))=x \quad (x\in \Q).
$$
(One has $\{\lim_{y\to x^+}\Jimm(x), \lim_{y\to x^-}\Jimm(x)\}=\{ \Jimm_\sharp(x), \Jimm_\flat(x)\}$ for all $x\in \Real$).
The restriction of $\Jimm$ to $\Real \backslash (\Q \cup \mathcal N)$ is then an involution, 
and is $\alpha$-equivariant under the $\pgl$-action. In other words, it satisfies everywhere the functional equations~(\ref{eq:jimmfe}) (see \cite{conumerator}).
The amount of jump of $\Jimm$ at $x$ equals $|\Jimm_\sharp(x)-\Jimm_\flat(x)|$.
In fact, for every irrational $x$, the limits below exists and are equal:
$$
\lim_{y\to x}\Jimm_\sharp(x)=\lim_{y\to x}\Jimm_\flat(x)=\lim_{y\to x}\Jimm(x).
$$

\begin{theorem} Let $\Jimm_\sharp$ and $\Jimm_\flat$  be the $\alpha$-equivariant functions with respect to the $\psl$-action defined above.
\begin{enumerate}
\item The representation $\Psi_r: \pgl\to  \mathsf{PGL}_2(\Z,q=r)$ is conjugate to Dyer's outer automorphism $\alpha$ if and only if 
$
r=-\varphi^2 $ or $ r=-\bar\varphi^2.$
More precisely, there exists 
$M\in  \mathsf{PSL}_2(\C)$ with 
$$
M U_r M^{-1}= U, \quad M K_r M^{-1}=K, \quad M V_r M^{-1}=UV
$$
if and only if $(r, M)$ is one of 
$$
\left(-\bar\varphi^{2}, \quad \frac{x+\varphi}{-\varphi x+\varphi^2}\right), \quad 
\left(-\varphi^2, \quad \frac{x+\bar\varphi}{-\bar\varphi x+\bar\varphi^2}\right)
$$
with
$$
M\circ \psi_{-\bar\varphi^{2}}=\Jimm_\sharp, \quad
M\circ \psi_{-\varphi^{2}}=\Jimm_\flat.
$$ 
\item The representation $\Psi_r^+: \pgl\to  \mathsf{PGL}_2(\Z,q=r)$ is conjugate to Dyer's outer automorphism $\alpha$ if and only if 
$
r=-\varphi^2 $ or $ r=-\bar\varphi^2.$
More precisely, there exists 
$M\in  \mathsf{PSL}_2(\C)$ with 
$$
M U_r^+ M^{-1}= U, \quad M K_r^+ M^{-1}=K, \quad M V_r^+ M^{-1}=UV
$$
if and only if $(r, M)$ is one of 
$$
\left(-\bar\varphi^{2}, \quad 
\frac{x-\omega}
{(1+\bar\varphi\omega)x-(\bar\varphi\omega+\bar\varphi+\omega)}\right), \quad 
\left(-\varphi^{2}, \quad 
\frac{x-\omega}
{(1+\varphi\omega)x-(\varphi\omega+\varphi+\omega)}\right).
$$
with
$$
M\circ \psi_{-\bar\varphi^{2}}^+=\Jimm_\sharp, \quad
M\circ \psi_{-\varphi^{2}}^+=\Jimm_\flat.
$$ 
\item 
The representation $\Psi_r^-: \pgl\to  \mathsf{PGL}_2(\Z,q=r)$ is conjugate to Dyer's outer automorphism $\alpha$ if and only if 
$
r=-\varphi^2 $ or $ r=-\bar\varphi^2.$
More precisely, there exists 
$M\in  \mathsf{PSL}_2(\C)$ with 
$$
M U_r^- M^{-1}= U, \quad M K_r^- M^{-1}=K, \quad M V_r^- M^{-1}=UV
$$
if and only if $(r, M)$ is one of 
$$
\left(-\bar\varphi^{2}, \quad 
\frac{x-\bar\omega}
{(1+\bar\varphi\bar\omega)x-(\bar\varphi\bar\omega+\bar\varphi+\bar\omega)}\right), \quad 
\left(-\varphi^{2}, \quad 
\frac{x-\bar\omega}
{(1+\varphi\bar\omega)x-(\varphi\bar\omega+\varphi+\bar\omega)}\right).
$$
with
$$
M\circ \psi_{-\bar\varphi^{2}}^-=\Jimm_\sharp, \quad
M\circ \psi_{-\varphi^{2}}^-=\Jimm_\flat.
$$ 
\end{enumerate}
\end{theorem}

\begin{proof}
\begin{enumerate}
\item 
Let $\psi$ be the $\psl$-equivariant quantization map with respect to the representation $\Psi$. So we have 
$$
\psi(1+x)=1+q\psi(x), \quad \psi(-1/x)=-1/q\psi(x)
$$
Now suppose
$$
f(x):=\frac{ a \psi(x)+b}{ c \psi(x)+ d } \iff  
\psi(x)=\frac{ d  f(x)-b}{- c  f(x)+ a }, \quad ( a   d -b c \neq 0)
$$
We want this $f$ to be an equivariant map satisfying 
$f(1+x)=1+1/f(x)$ and $f(-1/x)=-f(x)$ (these are satisfied by Jimm).
One has
\begin{align*}
f(1+x)
&=\frac{ a \psi(x+1)+ b }{ c \psi(x+1)+ d } 
=\frac{ a  q\psi(x)+ b + a }{ c  q\psi(x)+ d + c } =
\frac{ a  q\frac{ d  f(x)- b }{- c  f(x)+ a } + b + a }{ c  q\frac{ d  f(x)- b }{- c  f(x)+ a } + d + c } \\
&=\frac{ a  q({ d  f(x)- b })+({- c  f(x)+ a })( b + a )}
       { c  q({ d  f(x)- b })+({- c  f(x)+ a })( d + c )}\\
&=\frac{( a  q  d - c ( b + a ))f(x)+(- a  q b + a ( b + a ))}
{( c  q  d - c ( d + c ))f(x)+(- c  q b + a ( d + c ))},
\end{align*}
\begin{align}
f(-1/x)
&=\frac{ a -q b \psi(x)}{ c -q d \psi(x)} =
\frac{ a -q b \frac{ d  f(x)- b }{- c  f(x)+ a }}{ c -q d \frac{ d  f(x)- b }{- c  f(x)+ a }}
=\frac{ a (- c  f(x)+ a )-q b ( d  f(x)- b )}{ c (- c  f(x)+ a )-q d ( d  f(x)- b )}\\
&=
\frac
{-( a  c +q b  d ) f(x)+( a ^2+q b ^2)}
{-( c ^2+q d ^2) f(x)+( a  c +q d  b )}
\end{align}
So the equations $f(1+x)=1+1/f(x)$ and $f(-1/x)=-f(x)$ imposes
$$
 a ^2+q b ^2= c ^2+q d ^2=- c  q b + a ( d + c )=0
$$
$$
 a  q  d - c ( b + a )=- a  q b + a ( b + a )=( c  q  d - c ( d + c ))
$$
This system admits the solution
$$
f(x) = \frac{\psi(x)+\varphi}{-\varphi \psi(x)+\varphi^2}
\mbox{ with } q=-\bar\varphi^2
$$
and its conjugate
$$
\bar f(x) = \frac{\psi(x)+\bar\varphi}{-\bar \varphi \psi(x)+\bar\varphi^2}
\mbox{ with } q=-\varphi^2.
$$
It is routine to check that $f$ and $\bar f$ satisfies the other functional equations of $\Jimm$, 
i.e. $f(-x)=-1/f(x)$, $f(1/x)=1/f(x)$ and $f(1-x)=1-f(x)$. 

Note that both $M$'s are in $\pslr$ and can be normalized by dividing with $\sqrt{2}\varphi$ or $\sqrt{2}\bar\varphi$. Also note that both $M$'s  has $\omega$, $\bar\omega$ as their fixed points.
\item The proof is similar to the first case.
\item The proof is similar to the first case.
\end{enumerate}
\end{proof}
Observe that $(-\bar \varphi^2)(\varphi^2)=1$, reflecting the symmetry $q\leftrightarrow   1/q$ discussed in~\cite{genoud}.
This pair of numbers appear in several contexts in the recent paper~\cite{etingof}, too.

For sake of clarity, let us explicitly describe the target sets of the maps $\Psi_r$ discussed above:
\begin{align*}
\mathsf{PSL}_2(\Z, {q=-\bar \varphi^2})=\left\langle 1-\bar \varphi^2x, \quad \frac{1}{\bar\varphi^{2}x} \right \rangle
=\left\langle 1-\frac{1}{x}, \quad \frac{1}{\bar\varphi^{2}x} \right \rangle
<\mathsf{PGL}_2(\Real),\\
\mathsf{PSL}_2(\Z, {q=-\varphi^2})=\left\langle 1-\varphi^2x, \quad \frac{1}{\varphi^{2}x} \right \rangle
=\left\langle 1-\frac{1}{x}, \quad \frac{1}{\varphi^{2}x} \right \rangle
<\mathsf{PGL}_2(\Real).
\end{align*}
Since $r=-\varphi^2,-\bar \varphi^2<0$ we have
$\Psi_r(1+x)= x:\mapsto1+rx \notin \pslr$ and  $\Psi_r(1/x)=x\mapsto-1/rx\notin \pslr$. 
Therefore the images of the representations $\Psi_{-\varphi^2}$
$\Psi_{-\bar \varphi^2}$,  are not contained inside $\pslr$. 
We have the exact sequences (note $\bar\varphi=-1/\varphi$)
\begin{align*}
1\to \Psi_{-\varphi^{\pm2}}(\Gamma) &\to \mathsf{PSL}_2(\Z, {q=-\varphi^{\pm 2}}) \to \langle \pm 1 \rangle\to 1,
\end{align*}
where $\Gamma<\psl$ is the subgroup $\langle L, SLS \rangle$ discussed above, and the surjection is the 
projective determinant.
To see the kernels of the exact sequence above clearly as subgroups of $\pslr$, let us  describe them explicitly in matrix form:
$$
\Psi_{-\varphi^2}(\Gamma)= 
\left\langle 
\begin{bmatrix} 1&-1\\0&1 \end{bmatrix}, \,
\begin{bmatrix} 0&\varphi^{-2}\\ -\varphi^2&1 \end{bmatrix}
\right\rangle,
\quad 
\Psi_{-\bar\varphi^2}(\Gamma)= 
\left\langle 
\begin{bmatrix} 1&-1\\0&1 \end{bmatrix}, \,
\begin{bmatrix} 0&\varphi^{2}\\ -\varphi^{-2}&1 \end{bmatrix}
\right\rangle.
$$
(We can make the groups $\mathsf{PSL}_2(\Z, q=-\varphi^{\pm2})$ act on the upper half plane,
by modifying $\Psi$ via
$\Psi(1+z):=1+q  z^*$, $\Psi(1/z):=-1/q z^*$, where $z^*$ is the complex conjugate of $z$).
A fundamental region for $\mathsf{PSL}_2(\Z, q=-\varphi^{\pm2})$ can be found as the pull-back by $M$ of the fundamental region of $\langle L, SLS \rangle <\psl$.

Note that there do exist $\alpha$-equivariant meromorphic 
functions on the upper half plane with respect to the $\Gamma$-action
\cite{conumerator}. The Schwarzian of an equivariant function is weight-4 modular form.

\newpage
\thispagestyle{empty}
\begin{table}[H]
\centering
\begin{tabular}{c c c c}
$r=2$ & \fbox{\includegraphics[width=0.3\linewidth]{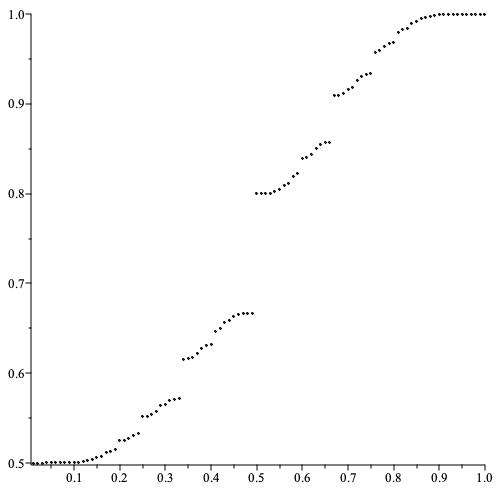}} & \fbox{\includegraphics[width=0.3\linewidth]{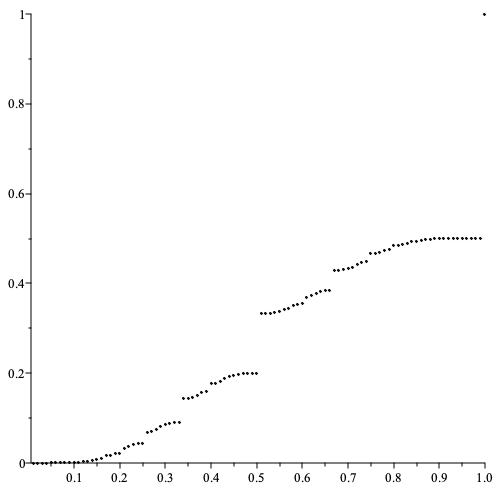}} & $r=1/2$ \\[10pt]
$r=3/2$ & \fbox{\includegraphics[width=0.3\linewidth]{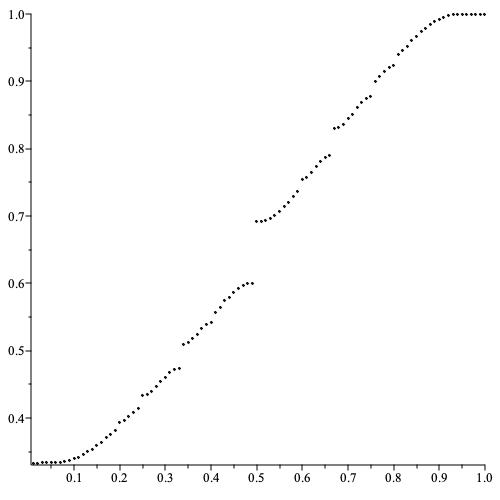}} & \fbox{\includegraphics[width=0.3\linewidth]{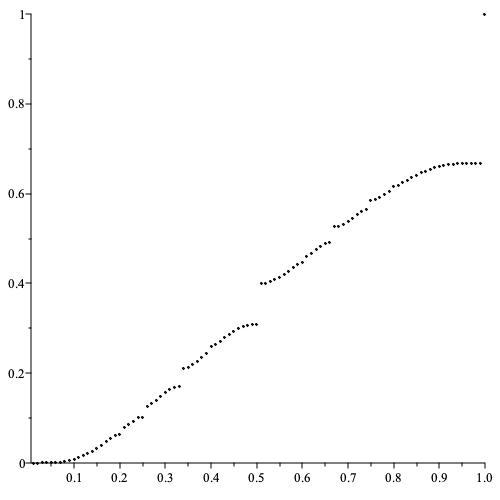}} & $r=2/3$ \\[10pt]
$r=4/3$  & \fbox{\includegraphics[width=0.3\linewidth]{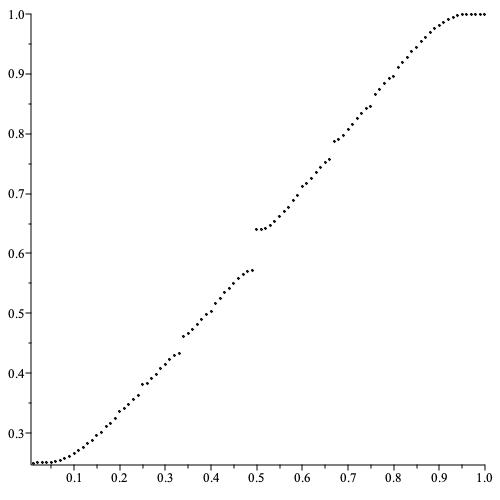}} & \fbox{\includegraphics[width=0.3\linewidth]{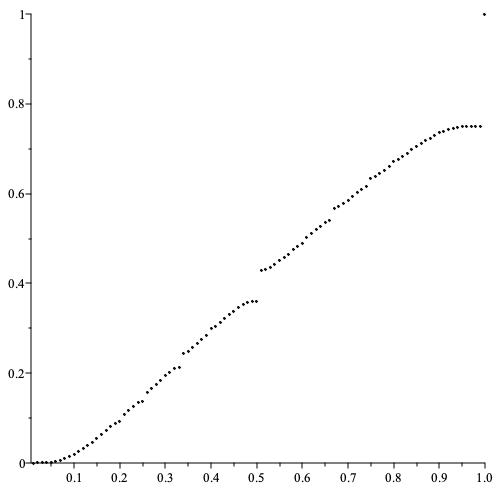}} & $r=3/4$ \\[10pt]
$r=5/4$  & \fbox{\includegraphics[width=0.3\linewidth]{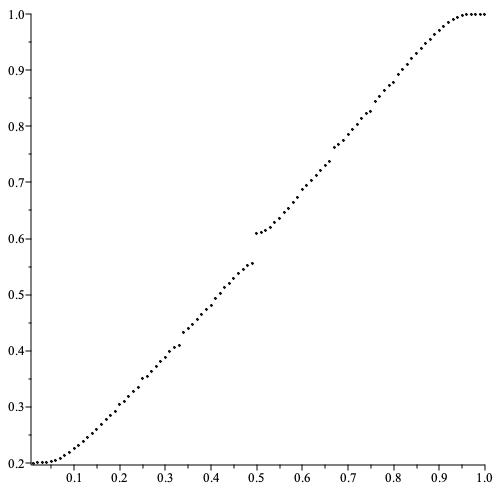}} & \fbox{\includegraphics[width=0.3\linewidth]{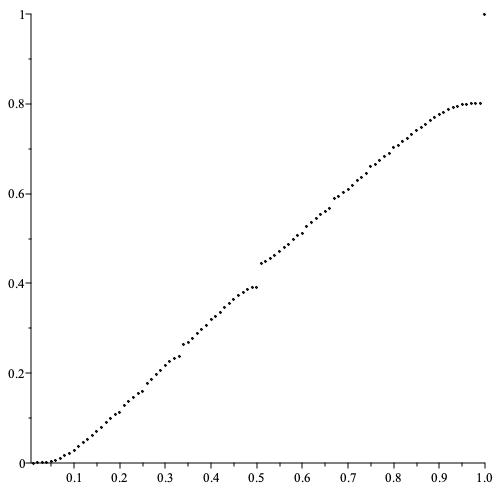}} & $r=4/5$ \\[10pt]
\end{tabular}
\caption{Plots of $\psi_r$ for some positive real values of $r$.}
\label{tab:roots_figures}
\end{table}

\newpage
\thispagestyle{empty}
\begin{table}[H]
\centering
\begin{tabular}{c c c c}
$r=-2$ & \fbox{\includegraphics[width=0.3\linewidth]{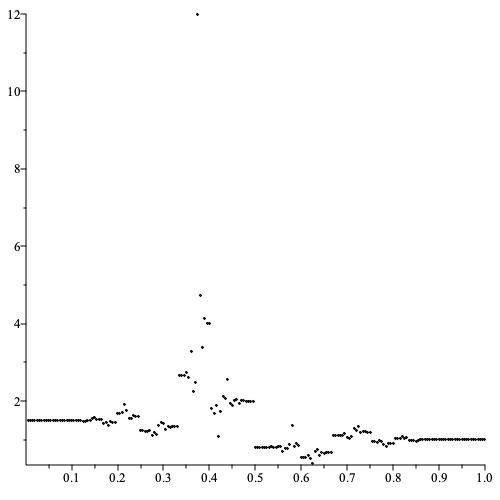}} & \fbox{\includegraphics[width=0.3\linewidth]{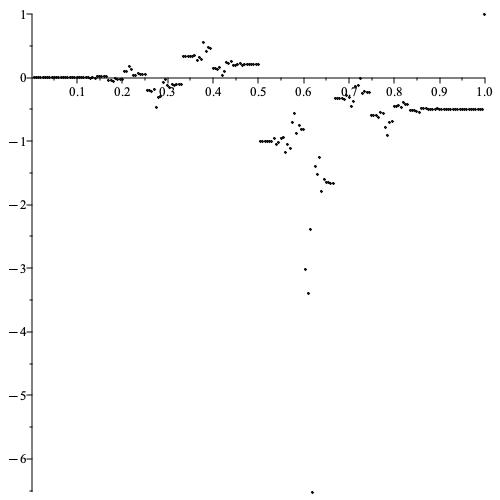}} & $r=-1/2$ \\[10pt]
$r=-3$ & \fbox{\includegraphics[width=0.3\linewidth]{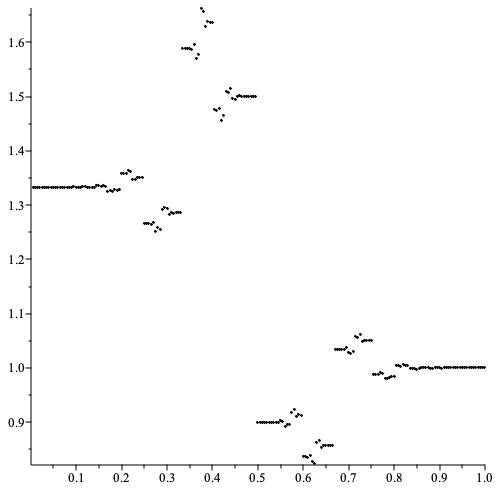}} & \fbox{\includegraphics[width=0.3\linewidth]{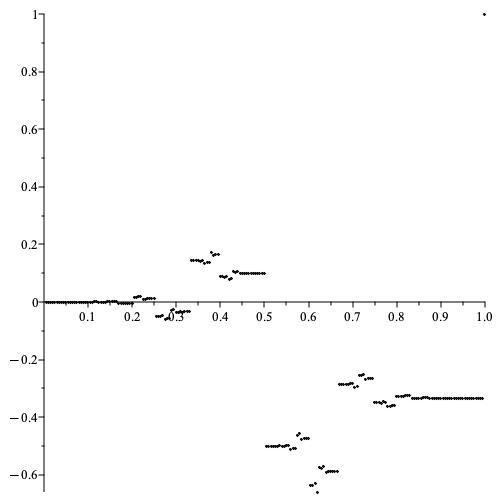}} & $r=1/3$ \\[10pt]
$r=-4$  & \fbox{\includegraphics[width=0.3\linewidth]{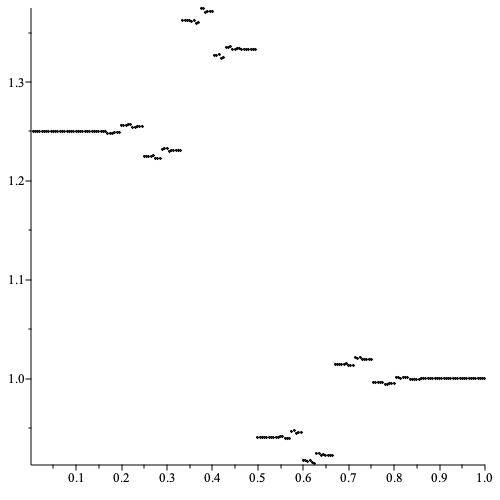}} & \fbox{\includegraphics[width=0.3\linewidth]{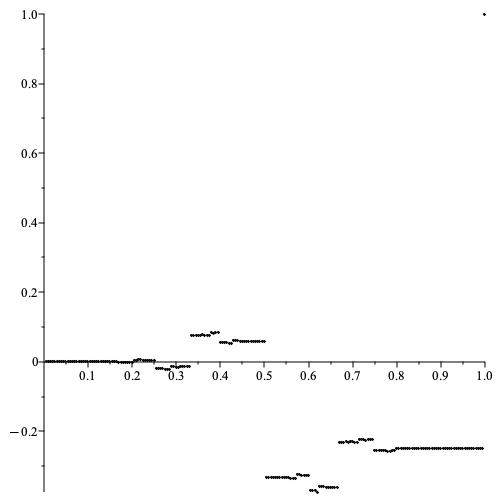}} & $r=-1/4$ \\[10pt]
$r=-5$  & \fbox{\includegraphics[width=0.3\linewidth]{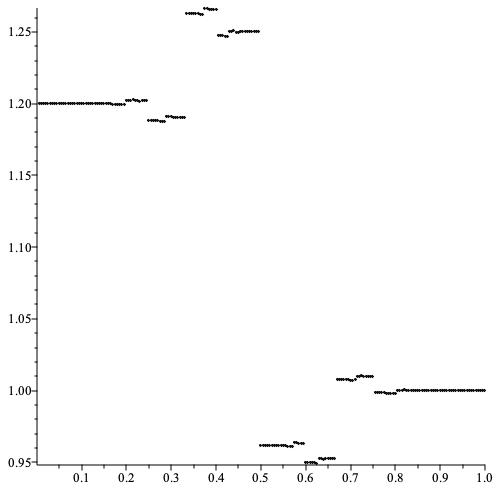}} & \fbox{\includegraphics[width=0.3\linewidth]{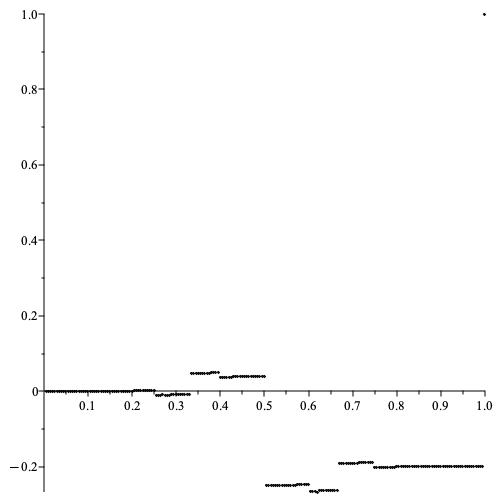}} & $r=-1/5$ \\[10pt]
\end{tabular}
\caption{Plots of $\psi_r$ for some negative real values of $r$.}
\label{tab:roots_figures2}
\end{table}

\newpage
\thispagestyle{empty}
\begin{figure}[H]
\centering
\fbox{\includegraphics[width=0.35\linewidth]{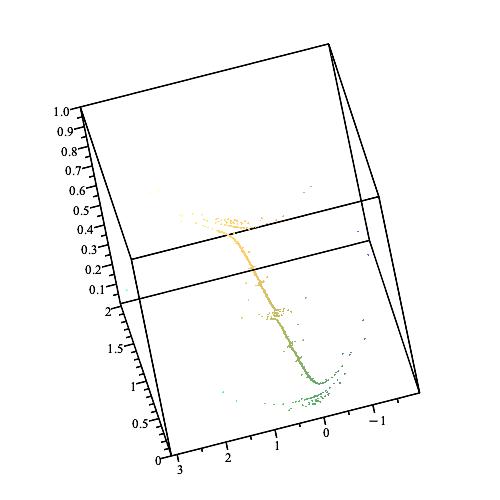}}
\caption{Plot of $\Psi_r$ at $r=\exp\left(\frac{2\pi i}{17}\right)$}
\end{figure}

\thispagestyle{empty}
\begin{figure}[H]
\centering
\fbox{\includegraphics[width=0.35\linewidth]{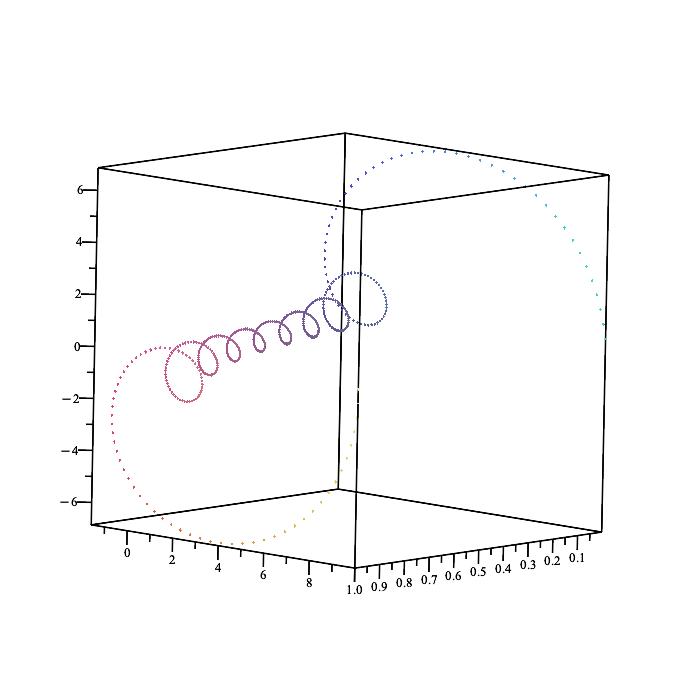}}
\caption{Plot of $\psi_r(10)$ with $x=10$ fixed while $r$ traces the unit circle.  }
\end{figure}

\begin{figure}[H]
\centering
\fbox{\includegraphics[width=0.35\linewidth]{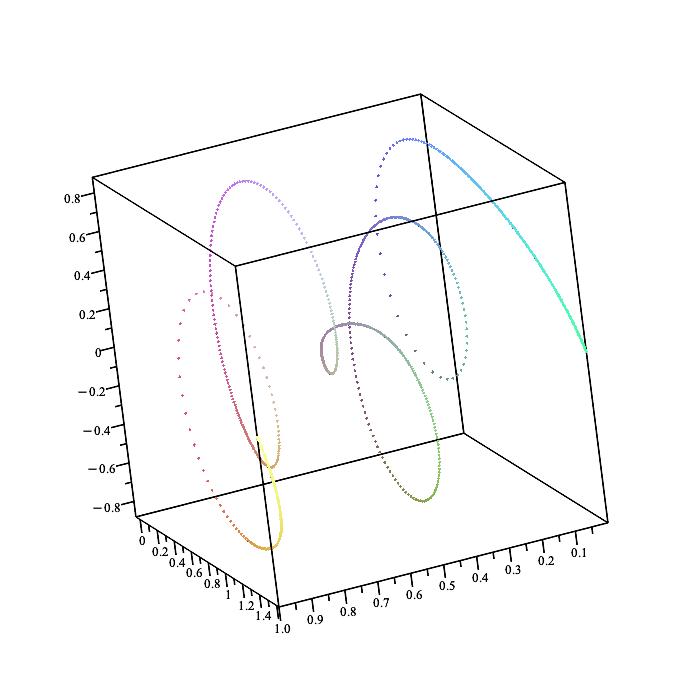}}
\caption{Plot of $\psi_r([1,1,1,1,1,1])$ with $x=[1,1,1,1,1,1]$ fixed while $r$ traces the unit circle. }
\end{figure}


\begin{thebibliography}{}
%
%
\bibitem{genoud} 
S. Morier-Genoud and V. Ovsienko. 
{\it $ q $-deformed rationals and irrationals.}
arXiv preprint arXiv:2503.23834 (2025).


\bibitem{Valentin} 
S. Morier Genoud and V. Ovsienko.
{\it On q-deformed real numbers.} 
Experimental Mathematics  2022, Vol. 31, No. 2, 652–660.

\bibitem{etingof}
P. Etingof, 
{\it On $ q $-real and $ q $-complex numbers.}
 arXiv preprint arXiv:2508.08440 (2025).
 
\bibitem{jouteur}
P. Jouteur.
{\it Symmetries of the q-deformed real projective line.} 
arXiv preprint arXiv:2503.02122 (2025).


\bibitem{conumerator}
A. M. Uluda\u{g} and B. Eren G\"{o}kmen.
{\it The Conumerator and the Codenominator.} 
Bulletin des Sciences Mathématiques,
Volume 180, November 2022, 103192.

  
  
\bibitem{Uludag}
A. M. Uluda\u{g} and  H. Ayral. 
{\it On the Involution Jimm.} in: 
IRMA Lectures in Mathematics and Theoretical Physics, 
2021, Vol. 33, pp. 561-578.

\bibitem{uludagg}
A. M. Uludağ,{\it On the involution Jimm.}
Topology and geometry–a collection of essays dedicated to Vladimir G. Turaev: 561-578.


\bibitem{zeytin1}
A. M. Uludağ, A. Zeytin and M. Durmuş.
{\it Binary quadratic forms as dessins.}
Journal de théorie des nombres de Bordeaux 29.2 (2017): 445-469.

\bibitem{zeytin2}
A. M. Uludağ and A. Zeytin.
{\it A panaroma of the fundamental group of the modular orbifold.}
Handbook of Teichmüller theory 6 (2016): 501-519.





\end{thebibliography}
\end{document}